\documentclass[preprint]{article}
\usepackage{lineno}
\usepackage{tikz}
\usetikzlibrary{arrows}
\usepackage{graphicx}
\usepackage{color}
\usepackage[colorlinks]{hyperref}
\usepackage{euscript}
\usepackage{mathdots}
\usepackage{yhmath}
\usepackage{cancel}
\usepackage{siunitx}
\usepackage{array}
\usepackage{multirow}
\usepackage{gensymb}
\usepackage{tabularx}
\usepackage{booktabs}
\usepackage{acronym}
\usepackage{enumitem}
\usepackage{upgreek}
\usepackage[all]{xy}
\usepackage{mathtools}
\usepackage[normalem]{ulem}
\usepackage{amsmath,amsthm,amsfonts,amssymb,amscd,euscript}
\usepackage{inputenc}

\newcommand{\R}{\mathbb R}
\newcommand{\Z}{\mathbb Z}
\newcommand{\N}{\mathbb N}

\newcommand{\ifs}{$\mathfrak{I}=(X,\,\mathcal{F},\,\Sigma)\;$}

\newtheorem{thm}{Theorem}[section]

\newtheorem{prop}[thm]{Proposition}
\newtheorem{example}[thm]{Example}
\newtheorem{defn}[thm]{Definition}
\theoremstyle{remark}
\newtheorem{rem}[thm]{Remark}

\makeatother

\begin{document}

\title{Iterated function systems over arbitrary shift spaces}
\author{Dawoud Ahmadi Dastjerdi\footnote{dahmadi1387@gmail.com (ahmadi@guilan.ac.ir)},\and Mahdi Aghaee \footnote{mahdi.aghaei66@gmail.com}
}
\maketitle

\begin{abstract}
The orbit of a point $x\in X$  in a classical iterated function system (IFS) can be defined as
$\{f_u(x)=f_{u_n}\circ\cdots \circ f_{u_1}(x):$ $u=u_1\cdots u_n$ is a word of a full shift $\Sigma$ on finite symbols and $f_{u_i}$ is a continuous self map on $X$  $\}$.
One also can associate to $\sigma=\sigma_1\sigma_2\cdots\in\Sigma$ a non-autonomous system $(X,\,f_\sigma)$ where the trajectory of $x\in X$ is defined as $x,\,f_{\sigma_1}(x),\,f_{\sigma_1\sigma_2}(x),\ldots$.
 Here instead of the full shift, we consider 
an arbitrary shift space $\Sigma$. Then we investigate  basic properties related to this IFS and the associated non-autonomous systems.
In particular, we look for sufficient conditions that guarantees that in a transitive IFS one may have a transitive 
$(X,\,f_\sigma)$ for some $\sigma\in\Sigma$  and how abundance are such $\sigma$'s.
\end{abstract}


\section{Introduction}

In a classical dynamical system, here called \emph{conventional dynamical system}, we have a phase space and a unique  map where the trajectories of points are obtained by iterating this map. 
However, in  various problems, including applied ones, one may have some finite sequence of maps in place of a single map acting on the same phase space.
As  an example let $X$ be the space of mixture of some materials which are supposed to be mixed by application of two robotic arms $r_0$ and $r_1$ and only one of them at each unit of time. Due to some technical considerations, two $r_1$ cannot be applied in a row, though this consideration is not on place for $r_0$. Thus the application of these arms, and hence the dynamics of the system, is bound to the golden subshift, i.e., the subshift whose forbidden set is $\{11\}$. 
In fact, there are many natural processes whose evolution evolve with discrete time which are  involved with two or more interactions. For instance, in Physics by two or more maps have appeared in \cite{allison2001control,parrondonew}, Economy in \cite{zhang2006discrete} and Biology in \cite{de2008generalized}. In Mathematics, this has been studied either by non-autonomous systems in many literature such as \cite{kolyada1996topological} or as iterated function system (IFS) 
for constructing and studying some fractals
in \cite{diaconis1999iterated,hutchinson1981fractals} or for investigating dynamical properties in some studies such as \cite{barnsley2013conley,barrientos2017chaos,glavan2006shadowing,hui2018some}.

In a ``classical" IFS, a compact metric space $X$ and a set of some $k$ finite continuous functions $\{f_0,\,\cdots,\, f_{k-1}\}$ on $X$ are assumed and the trajectory of a point $x\in X$ is considered to be the action on $x$ of the sequence  of freely combination of those maps, or action on $x$ of combination of those maps over the words of a full shift: just write $f_u=f_{u_1}\circ\cdots\circ f_{u_m}$ where $u=u_1\cdots u_m$ is a word of the full shift over $k$ symbols. Hence no limitation is applied  as in our aforesaid example on the robotic arms where there words were forbidden to have $11$ as a subword. 
The limitation applied on the shift space would transfer to some limitations on the system.
For instance, a system may be {topological transitive} in classical IFS but not in our case, i.e., when the full shift is replaced with a more general subshift.

Thus one may look at $X$ as a phase space and the subshift $\Sigma$ as a  parameter space showing how the maps must be combined.

\vskip.5cm
	\noindent{\bf The present paper is organized as follows.} 
	In Section \ref{preliminary}, we formalize the definitions and notations. 
	Section \ref{sec transitivity} is mainly devoted to the definitions of transitivity in IFS and relation between them. In particular, we show that when the shift space is sofic, topological transitivity in the constituent IFS implies the {point transitivity} along a transitive orbit in the shift space; a fact which is not necessarily satisfied for nonsofics. 
	In Section \ref{subs st of S}, we like to see how large the set $S=\{\sigma\in\Sigma:\;\exists\, x\in X,\,\overline{ \mathcal{O}_{\sigma}(x)}=X\}$ can be. Section \ref{sec Mixing in IFS } mixing and exactness of an IFS versus to those properties along orbits through some examples has been considered.

\section{Preliminaries}\label{preliminary}

\subsection{Iterated function systems}
Throughout the paper, $ X $ will be a compact
metric space. 
The \emph{classical} iterated function system (IFS) consists of finitely many continuous self maps $ \mathcal{F}=\left\lbrace f_0,\,\ldots,\,f_{k-1} \right\rbrace  $ on  $X$. 
The \emph{forward orbit} of a point $x\in X$, denoted by $\mathcal{O}^+(x)$, is the set of all values of  finite possible combinations of $f_i$'s at $x$. 
We need the following equivalent statement: Let $\Sigma_{|\mathcal{F}|}$ be the full shift on $k$ symbols and let $\mathcal{L}(\Sigma_{|\mathcal{F}|})$ called the \emph{language of $\Sigma_{|\mathcal{F}|}$} be the set of words or blocks.
Define $f_u(x):X\rightarrow X$ by 
\begin{equation}\label{eq def of f_u}
	f_{u_n}\circ\cdots \circ f_{u_1}(x),\quad u=u_1\cdots u_n\in\mathcal{L}(\Sigma_{|\mathcal{F}|}).
\end{equation}
Then $\mathcal{O}^+(x)=\{f_u(x):\; u\in \mathcal{L}(\Sigma_{|\mathcal{F}|})\}$.
Such iterated function systems, here called \emph{classical IFS}, have 
been the subject of study for quite a long time.

Here we define an IFS to be
\begin{equation}\label{eq IFS}
	\mathfrak{I}=(X,\,  \mathcal{F}=\left\lbrace f_0,\,\ldots,\, f_{k-1}\right\rbrace,\, \Sigma).
\end{equation}
where each $f_i$ is continuous and $\Sigma$ is an arbitrary subshift on $k$ symbols, not necessarily the full shift $\Sigma_{|\mathcal{F}|}$ as in the classical IFS.
Later a brief review of symbolic dynamics will be given  in subsection \ref{subsec Symbolic}. 
By this setting, $\Sigma_{|\mathcal{F}|}$ above will be replaced with $\Sigma$ and thus
$\mathcal{O}^+(x)=\{f_u(x):\; u\in \mathcal{L}(\Sigma)\}$ is the forward orbit of $x$. In particular, $f_u(f_v(x))=f_{vu}(x)$  whenever $vu$ is admissible or equivalently $vu\in\mathcal{L}(\Sigma)$.
Let $u=u_1\cdots u_n\in\mathcal{L}(\Sigma)$ and set $u^{-1}:=u_n\cdots u_1$.
Then for $A\subseteq X$,
\begin{eqnarray*}
	(f_{u})^{-1}(A)&=&(f_{u_n}\circ\cdots\circ f_{u_1})^{-1}(A)\\
	&=&f^{-1}_{u_1}\circ\cdots\circ f^{-1}_{u_n}(A)\\
	&=&f^{-1}_{u^{-1}}(A),
\end{eqnarray*}
where for the last equality, we used \eqref{eq def of f_u}. Also
\begin{eqnarray*}
	f^{-1}_{u^{-1}}( f^{-1}_{v^{-1}}(A))&=&f^{-1}_{v^{-1}u^{-1}}(A)=f^{-1}_{(uv)^{-1}}(A)\\
	&=&\left(f_{uv}\right)^{-1}(A).
\end{eqnarray*}
Thus the backward orbit and the (full) orbit of a point $x\in X$ are $\mathcal{O}_-(x)=\{f^{-1}_{u^{-1}}(x):\; u\in\mathcal{L}(\Sigma)\}$ and $\mathcal{O}(x)=\mathcal{O}^+_-(x)=\mathcal{O}^+(x)\cup \mathcal{O}_-(x)$ respectively.

When all $f_i$'s are homeomorphisms, the backward, forward and full trajectory of $x$ is defined.

We say $\mathcal{F}=\lbrace f_0,\, \ldots ,\,f_{k-1}\rbrace$  is \emph{surjective},
\emph{injective}, \emph{homeomorphism} if all $f_i$'s in $\mathcal{F}$  are so.

When $k=1$ and $ \Sigma=\{0^{\infty}\} $,  we simply have the classical dynamical system, here called \emph{conventional dynamical system} denoted either by the pair $ (X,\, f ) $ or $\mathfrak{I}=(X,\,\{f_0\},\,\{0^\infty\})$.
\subsection{Symbolic dynamics}\label{subsec Symbolic}
A brief recall of the symbolic dynamics is given here.
Notations and main ideas are borrowed from \cite{lind2021introduction} and the proofs of the claims  can be found there.
Let $ \mathcal{A} $ be a non-empty finite set and let  $\Sigma_{|\mathcal{A}|}= \mathcal{A}^{\Z} $ (resp.  $ \mathcal{A}^{\N}  $) be the collection of all bi-infinite (resp. right-infinite) sequences of symbols from $ \mathcal{A} $. The shift map on $ \Sigma_{|\mathcal{A}|} $ is the map $ \tau $  where $ \tau(\sigma)=\sigma' $ is defined by  $ \sigma'_i=\sigma_{i+1} $.
 The pair $ (\Sigma_{|\mathcal{A}|},\,\tau) $ is the \emph{full shift} and any closed invariant subset $ \Sigma $ of $ \Sigma_{|\mathcal{A}|} $ is called a \emph{subshift} or a \emph{shift space}.
A \emph{word}  or \emph{block} over $ \mathcal{A} $ is a finite sequence of symbols from $ \mathcal{A} $. 
Denote by $ \mathcal{L}_{n}(\Sigma) $ 
the set of all admissible $ n $-words and call 
$ \mathcal{L}(\Sigma): =\bigcup_{n=0}^{\infty}\mathcal{L}_n(\Sigma)$ 
the \emph{language} of $ \Sigma $. 
For $ u \in \mathcal{L}_k(\Sigma) $,
let the cylinder $_\ell[u]_{\ell+k-1}$ $= _\ell\hspace{-1mm}[u_\ell\cdots u_{\ell+k-1}]_{\ell+k-1} $ be the set $ \{\sigma=\cdots \sigma_{-1}\sigma_0\sigma_1\cdots \in \Sigma :
\sigma_{\ell}\cdots\sigma_{\ell+k-1} = u\} $. If $\ell=0$, we drop the subscripts and we just write $[u]$.

A shift space $ \Sigma $ is \emph{irreducible} if for every ordered pair of words $ u,\,v\in \mathcal{L} (\Sigma) $ there is a word $ w\in \mathcal{L} (\Sigma) $ so that $ uwv\in \mathcal{L} (\Sigma) $.
A point $ \sigma \in \Sigma $ is \emph{transitive} if every word in $ \Sigma $ appears in $ \sigma $ infinitely many often. 
A subshift $ \Sigma $ is irreducible iff $ \Sigma $ has a transitive point.

Shift spaces described by a finite set of forbidden blocks are called \emph{shifts of finite type} (SFT) and their factors are called \emph{sofic}. A word $ w\in \mathcal{L}(\Sigma) $ is called \emph{synchronizing} if $ uwv \in\mathcal{L}(\Sigma) $ whenever $uw,wv \in\mathcal{L}(\Sigma)$. A \emph{synchronized system} is an irreducible shift which has a synchronizing word. Any sofic is synchronized.

A subshift $ \Sigma $ is \emph{specified}, or has the
\emph{specification property}, if there is $ N\in\N $ such that if $ u,v \in \mathcal{L}(\Sigma) $, then there is $ w $ of length $ N $ so that $ uwv \in \mathcal{L}(\Sigma) $.
A specified system is mixing and synchronized and any mixing sofic is specified.
A \emph{coded system} is the closure of the set of sequences obtained by freely concatenating the words in a list of words. In particular, any synchronized system is coded.

All synchronized systems have an (edge) labeled graph presentation called \emph{cover}. 	These  are directed graphs  whose edges are with assigned labels from $\mathcal{A}$ and infinite walk on the graph and recording the labels will represent a point in the subshift. 
The set of all such points is dense in the subshift. 
\subsubsection{Factors and extensions in an IFS}

There is a natural way to define factors in a non-autonomous  and in classical iterated function systems  \cite{kolyada1996topological} and \cite{liu2020topological}  Let $(X,\, f_{1,\,\infty})$ and $(Y,\, g_{1,\,\infty})$ be two non-autonomous systems. 
Then, $(Y,\, g_{1,\,\infty})$ is a factor of  $(X,\, f_{1,\,\infty})$, if there is a surjective continuous function $\varphi:X\rightarrow Y$ such that  $\varphi\circ f_i(x)=g_i\circ\varphi(x)$ for each $x\in X$ and each $i\in\N$. Also let  $\mathfrak{I}=(X,\,\mathcal{F},\, \Sigma_{|\mathcal{A}|})$ and 
$\mathfrak{I}'=(Y,\,\mathcal{G},\, \Sigma_{|\mathcal{A}'|})$ be two classical IFS where $\Sigma_{|\mathcal{A}|}$ and $\Sigma_{|\mathcal{A}'|}$ are full shifts over the finite alphabets $\mathcal{A}$ and $\mathcal{A}'$.
Then, $\mathfrak{I}'$ is a factor of $\mathfrak{I}$ if $|\mathcal{A}|\geq |\mathcal{A}'|$ and there are surjective continuous maps $\Psi:\Sigma_{|\mathcal{A}|}\rightarrow \Sigma_{|\mathcal{A}'|}$ and $\varphi:X\rightarrow Y$ such that $\varphi\circ f_{\sigma_i}(x)=g_{\Psi(\sigma_i)}\circ\varphi(x)$ for each $x\in X$ and each $\sigma\in\Sigma_{|\mathcal{A}|}$.
In the latter, by the way it has been defined, a necessity for  $\mathfrak{I}'$ being a factor of $\mathfrak{I}$ is that  $\Sigma_{|\mathcal{A}'|}$ being a factor of $\Sigma_{|\mathcal{A}|}$.

Now we set up to define factors in general IFS. First let $ \Sigma $ and $ \Sigma' $ be subshifts on the alphabets $ \mathcal{A} $ and $ \mathcal{A}' $ respectively and suppose that $ \Sigma'$ is a factor of $ \Sigma $ via an $ (m+n+1) $-block map $ \Psi:\mathcal{B}_{m+n+1}(\Sigma)\to \mathcal{A} $. Let $ \psi = \Psi_{\infty}=\Psi^{[-m,\,n]}_{\infty}:\Sigma\to \Sigma' $ be the induced sliding block map with 
\begin{equation*}
	\psi(\sigma) = \psi (\cdots \sigma_{-1}\sigma_{0}\sigma_{1}\cdots)=\cdots \sigma'_{-1}\sigma'_{0}\sigma'_{1}\cdots
\end{equation*}
whenever $ \Psi(\sigma_{i-m}\sigma_{i-m+1}\cdots  \sigma_{i+n})=\sigma'_{i} $. For one sided shifts, let $m=0$.
If necessary, by passing to higher block shifts \cite[\S 1.4]{lind2021introduction}, and replacing $\Sigma$ and $\Sigma'$ with suitable conjugate subshifts, one may consider that $\Psi$ to be a 1-block map.
\begin{defn}\label{def factor}
	Let $ \varphi $ be a continuous map from $ X $ onto $ Y $. Then $ \mathfrak{I}'=(Y,\mathcal{G}=\{g_0,\,\ldots,\,g_{\ell'}\},\,\Sigma') $ is a \emph{factor} of $\mathfrak{I}=(X,\mathcal{F}=\{f_0,\,\ldots,\,f_{\ell}\},\,\Sigma) $ if $\Sigma'$ is a factor of $\Sigma$ as above and for all $ y\in Y $ and $ x\in \varphi^{-1}(y) $
	\begin{equation*}
		\varphi\circ f_{\sigma_{i-m}\cdots \sigma_{i+n}}(x)=g_{\sigma_i'}(y) .
	\end{equation*}
	Hence, $\mathfrak{I}'$ is a factor of $\mathfrak{I}$, if for each $\sigma\in\Sigma$, the non-autonomous system $(Y,\, g_{\psi(\sigma)}) $ is a factor of $(X,\,f_\sigma)$ via $\varphi:X\rightarrow Y$.
	In this situation, $\mathfrak{I}$ is called an \emph{extension} of $\mathfrak{I}'$. 
	If $\mathfrak{I}$ is also a factor of $\mathfrak{I}'$, then $\mathfrak{I}$ and $\mathfrak{I}'$ are \emph{conjugate}.
\end{defn}
If necessary, by passing to an $N$-higher block shifts for some $N\in\N$ and replacing  $\Sigma$ with that new shift and $\mathcal{F}$ with 
$\{f_u:\; u \in\mathcal{L}_N(\Sigma)\}$, we may consider that $\Psi$ defining the block factor map between $\Sigma$ and $\Sigma'$, $N$-higher block shift, to be a 1-block map.

It is trivial to check that all dynamical properties defined in Definitions \ref{def intro} and \ref{defn along an orbit} are invariant under the factor map.

\section{Transitivity}\label{sec transitivity}
Two sorts of transitivity are very common in the study of topological dynamical systems: topological transitivity and point transitivity. These two concepts are the same for surjective conventional dynamical systems on the compact metric spaces such as subshifts but not for IFS's and non-autonomous dynamical systems.
Hence we say a transitive point in $\Sigma$ but will emphasize for point transitivity or topological transitivity in other places.
\begin{defn}\label{def intro}
	Consider $\mathfrak{I}$ as in \eqref{eq IFS} and let $U$ and $V$ be arbitrary open sets in $X$. Then $\mathfrak{I}$ is
	\begin{enumerate}
		\item 
		\emph{ ``forward" point transitive}, if  there is $x\in X$ such that $\{f_u(x):\;u\in\mathcal{L}(\Sigma)\}$ is dense in $X$. We drop ``forward" when it is clear from the context.
		\emph{Backward point transitivity} is likewise defined.
        \item
      \emph{topological transitive}, if there is $ u\in \mathcal{L}_n(\Sigma) $ such that $f_u(U)\cap V\neq\emptyset$.
		\item 
		\emph{mixing}, if there is $M=M(U,\,V)\in\N$ such that for $n\geq M$, there is $u\in\mathcal{L}_n(\Sigma)$ such that $f_u(U)\cap V\neq\emptyset$.
		\item 
		\emph{exact}, if there is $u(U)\in\mathcal{L}({\Sigma})$ such that for any $uu'\in\mathcal{L}(\Sigma)$, $f_{uu'}(U)=X$.
	\end{enumerate}
\end{defn}

We have the following implications in any IFS:
\begin{equation}\label{eq exmixtran}
	\text{exactness } \Rightarrow \text{ mixing } \Rightarrow \text{ topological transitivity } \Rightarrow \text {point transitivity}.
\end{equation}	
 The first two implications follow from the definition and the last from the next proposition. Also, since conventional dynamical systems are IFS,  they provides examples that the first two  implications are not reversible.
\begin{prop}\label{prop trans0}
	Let \ifs be a surjective IFS. If for arbitrary
	non-empty open sets  
	$U,\, V$ there is $u\in\mathcal{L}(\Sigma)$ such that
	$(f_u)^{-1}(U)\cap V \neq\emptyset$, then $\mathfrak{I}$ is point transitive.
	
\end{prop}
\begin{proof}
	Let $\mathcal{B} =\{U_n : n\in\N\} $ be a countable base for $ X $.
	Fix $n\in\N$ and set 
	\begin{equation}\label{eq G_n}
		G_n:=\cup_{u\in\mathcal{L}(\Sigma)}(f_{u})^{-1}(U_n).
	\end{equation} 
	By the assumption for an arbitrary open set $V$, $G_n\cap V\neq\emptyset$ and so  the open set $G_n$ is dense and as a result, $\cap_{n\in\N}G_n$ is residual.
	Hence for $x\in \cap_{n\in\N}G_n $ and any $n\in\N$, there is $u\in\mathcal{L}(\Sigma)$ such that 
	$ x\in (f_{u})^{-1}(U_n) $.
	This means 
	$ f_{u}(x)\in U_n $ and so $ x $ is a transitive point.
\end{proof}
Unlike a surjective conventional dynamical system, the last implication in \eqref{eq exmixtran} is not reversible. {This fact was noticed in some literature \cite{kontorovich2008note,parham2020iterated}; however, we did not find any example to justify, so we bring our own.} 
\begin{example}\label{Ex P.T not T.T}
Let $\mathfrak{I}=(X,\,\{f_0,\,f_1,\,f_2\},\,\Sigma_{|\mathcal{F}|})$ where $X=\{1/n:\;n\in\N\}\cup\{0\}\subset\R$ is equipped with subspace topology.

 Our maps are defined as follows. (See Figure \ref{fig P.T not T.T}.)
For all $i$, $f_i(0)=0$ and $f_i(1)=1$. 
\begin{eqnarray*}
	f_0(\frac{1}{n+1})&=&\frac{1}{n},\quad n\geq 1,\\
	f_1(\frac{1}{2n+1})&=&\frac{1}{2n},\quad n\geq 1,\\
	f_1(\frac{1}{2n})&=&\frac{1}{2n+1}, \quad n\geq 1
	\quad \text{and}\quad f_1(\frac{1}{2})=\frac{1}{3}.
\end{eqnarray*}
Also, $f_2(\frac{1}{4})=\frac{1}{2}$, $f_2(\frac{1}{2})=\frac{1}{3}$ and
\begin{eqnarray*}
	f_2(\frac{1}{2n+1})&=&\frac{1}{2n+2},\quad n\geq 1,\\
	f_2(\frac{1}{2n+2})&=&\frac{1}{2n+1}, \quad n\geq 2.
	\end{eqnarray*}
All maps are continuous, open and surjective. Both $f_1$ and $f_2$ are homeomorphism, but $f_0$ is not injective: $f_0(\frac{1}{2})=f_0(1)=1$.

Observe that any point $x=\frac{1}{n}$, $n\geq 2$ is transitive. However,  the system is not topological transitive. Because, for any $u$, $f_u(\{1\})\cap\{\frac{1}{2}\}=\emptyset$.
It is easy to observe that 
$\mathfrak{I}$ is neither topological nor point transitive along any orbit.

\begin{figure}
	
	\begin{center}
	\tikzset{every picture/.style={line width=0.40pt}} 
	
	\begin{tikzpicture}[x=0.50pt,y=0.50pt,yscale=-0.95,xscale=0.95]
	 [id:da17185013147064265]

\draw    (167,102) -- (219.89,102.32) ;
\draw [shift={(221.89,102.33)}, rotate = 180.35] [color={rgb, 255:red, 0; green, 0; blue, 0 }  ][line width=0.75]    (10.93,-3.29) .. controls (6.95,-1.4) and (3.31,-0.3) .. (0,0) .. controls (3.31,0.3) and (6.95,1.4) .. (10.93,3.29)   ;
\draw [shift={(167,102)}, rotate = 0.35] [color={rgb, 255:red, 0; green, 0; blue, 0 }  ][fill={rgb, 255:red, 0; green, 0; blue, 0 }  ][line width=0.75]      (0, 0) circle [x radius= 3.35, y radius= 3.35]   ;
\draw    (221.89,102.33) -- (274.78,102.65) ;
\draw [shift={(276.78,102.67)}, rotate = 180.35] [color={rgb, 255:red, 0; green, 0; blue, 0 }  ][line width=0.75]    (10.93,-3.29) .. controls (6.95,-1.4) and (3.31,-0.3) .. (0,0) .. controls (3.31,0.3) and (6.95,1.4) .. (10.93,3.29)   ;
\draw [shift={(221.89,102.33)}, rotate = 0.35] [color={rgb, 255:red, 0; green, 0; blue, 0 }  ][fill={rgb, 255:red, 0; green, 0; blue, 0 }  ][line width=0.75]      (0, 0) circle [x radius= 3.35, y radius= 3.35]   ;
\draw    (276.78,102.67) -- (329.67,102.99) ;
\draw [shift={(331.67,103)}, rotate = 180.35] [color={rgb, 255:red, 0; green, 0; blue, 0 }  ][line width=0.75]    (10.93,-3.29) .. controls (6.95,-1.4) and (3.31,-0.3) .. (0,0) .. controls (3.31,0.3) and (6.95,1.4) .. (10.93,3.29)   ;
\draw [shift={(276.78,102.67)}, rotate = 0.35] [color={rgb, 255:red, 0; green, 0; blue, 0 }  ][fill={rgb, 255:red, 0; green, 0; blue, 0 }  ][line width=0.75]      (0, 0) circle [x radius= 3.35, y radius= 3.35]   ;
\draw    (331.67,103) -- (384.56,103.32) ;
\draw [shift={(386.56,103.33)}, rotate = 180.35] [color={rgb, 255:red, 0; green, 0; blue, 0 }  ][line width=0.75]    (10.93,-3.29) .. controls (6.95,-1.4) and (3.31,-0.3) .. (0,0) .. controls (3.31,0.3) and (6.95,1.4) .. (10.93,3.29)   ;
\draw [shift={(331.67,103)}, rotate = 0.35] [color={rgb, 255:red, 0; green, 0; blue, 0 }  ][fill={rgb, 255:red, 0; green, 0; blue, 0 }  ][line width=0.75]      (0, 0) circle [x radius= 3.35, y radius= 3.35]   ;
\draw    (386.56,103.33) -- (439.44,103.65) ;
\draw [shift={(441.44,103.67)}, rotate = 180.35] [color={rgb, 255:red, 0; green, 0; blue, 0 }  ][line width=0.75]    (10.93,-3.29) .. controls (6.95,-1.4) and (3.31,-0.3) .. (0,0) .. controls (3.31,0.3) and (6.95,1.4) .. (10.93,3.29)   ;
\draw [shift={(386.56,103.33)}, rotate = 0.35] [color={rgb, 255:red, 0; green, 0; blue, 0 }  ][fill={rgb, 255:red, 0; green, 0; blue, 0 }  ][line width=0.75]      (0, 0) circle [x radius= 3.35, y radius= 3.35]   ;
\draw    (441.44,103.67) -- (545.33,103.99) ;
\draw [shift={(547.33,104)}, rotate = 180.18] [color={rgb, 255:red, 0; green, 0; blue, 0 }  ][line width=0.75]    (10.93,-3.29) .. controls (6.95,-1.4) and (3.31,-0.3) .. (0,0) .. controls (3.31,0.3) and (6.95,1.4) .. (10.93,3.29)   ;
\draw [shift={(441.44,103.67)}, rotate = 0.18] [color={rgb, 255:red, 0; green, 0; blue, 0 }  ][fill={rgb, 255:red, 0; green, 0; blue, 0 }  ][line width=0.75]      (0, 0) circle [x radius= 3.35, y radius= 3.35]   ;
\draw    (547.33,104) ;
\draw [shift={(547.33,104)}, rotate = 0] [color={rgb, 255:red, 0; green, 0; blue, 0 }  ][fill={rgb, 255:red, 0; green, 0; blue, 0 }  ][line width=0.75]      (0, 0) circle [x radius= 3.35, y radius= 3.35]   ;
\draw    (124.33,102) ;
\draw [shift={(124.33,102)}, rotate = 0] [color={rgb, 255:red, 0; green, 0; blue, 0 }  ][fill={rgb, 255:red, 0; green, 0; blue, 0 }  ][line width=0.75]      (0, 0) circle [x radius= 3.35, y radius= 3.35]   ;
\draw [shift={(124.33,102)}, rotate = 0] [color={rgb, 255:red, 0; green, 0; blue, 0 }  ][fill={rgb, 255:red, 0; green, 0; blue, 0 }  ][line width=0.75]      (0, 0) circle [x radius= 3.35, y radius= 3.35]   ;
\draw    (547.33,104) .. controls (653.89,5.33) and (655.89,170.33) .. (548.89,109.33) ;
\draw    (124.33,102) .. controls (51.01,53.08) and (66.4,131.9) .. (123.26,108.48) ;
\draw [shift={(125.89,107.33)}, rotate = 155.41] [fill={rgb, 255:red, 0; green, 0; blue, 0 }  ][line width=0.08]  [draw opacity=0] (10.72,-5.15) -- (0,0) -- (10.72,5.15) -- (7.12,0) -- cycle    ;
\draw    (547.33,104) .. controls (633.02,47.9) and (613.3,141.43) .. (552.74,112.26) ;
\draw [shift={(550.89,111.33)}, rotate = 27.3] [fill={rgb, 255:red, 0; green, 0; blue, 0 }  ][line width=0.08]  [draw opacity=0] (10.72,-5.15) -- (0,0) -- (10.72,5.15) -- (7.12,0) -- cycle    ;
\draw    (547.33,104) .. controls (666.89,-36.67) and (694.89,194.33) .. (548.89,109.33) ;
\draw    (124.33,102) .. controls (5.89,14.33) and (27.89,145.33) .. (125.89,107.33) ;
\draw    (124.33,102) .. controls (4.89,-25.67) and (-28.11,163.33) .. (125.89,107.33) ;
\draw    (221.89,102.33) .. controls (211.11,64.11) and (185.92,63.35) .. (168.08,99.73) ;
\draw [shift={(167,102)}, rotate = 294.83] [fill={rgb, 255:red, 0; green, 0; blue, 0 }  ][line width=0.08]  [draw opacity=0] (10.72,-5.15) -- (0,0) -- (10.72,5.15) -- (7.12,0) -- cycle    ;
\draw    (276.78,102.67) .. controls (266,64.45) and (240.81,63.68) .. (222.97,100.06) ;
\draw [shift={(221.89,102.33)}, rotate = 294.83] [fill={rgb, 255:red, 0; green, 0; blue, 0 }  ][line width=0.08]  [draw opacity=0] (10.72,-5.15) -- (0,0) -- (10.72,5.15) -- (7.12,0) -- cycle    ;
\draw    (331.67,103) .. controls (320.89,64.78) and (295.7,64.02) .. (277.86,100.39) ;
\draw [shift={(276.78,102.67)}, rotate = 294.83] [fill={rgb, 255:red, 0; green, 0; blue, 0 }  ][line width=0.08]  [draw opacity=0] (10.72,-5.15) -- (0,0) -- (10.72,5.15) -- (7.12,0) -- cycle    ;
\draw    (386.56,103.33) .. controls (375.78,65.11) and (350.59,64.35) .. (332.75,100.73) ;
\draw [shift={(331.67,103)}, rotate = 294.83] [fill={rgb, 255:red, 0; green, 0; blue, 0 }  ][line width=0.08]  [draw opacity=0] (10.72,-5.15) -- (0,0) -- (10.72,5.15) -- (7.12,0) -- cycle    ;
\draw    (441.44,103.67) .. controls (430.66,65.45) and (405.48,64.68) .. (387.64,101.06) ;
\draw [shift={(386.56,103.33)}, rotate = 294.83] [fill={rgb, 255:red, 0; green, 0; blue, 0 }  ][line width=0.08]  [draw opacity=0] (10.72,-5.15) -- (0,0) -- (10.72,5.15) -- (7.12,0) -- cycle    ;
\draw    (167,102) .. controls (177.62,140.35) and (205.45,134.65) .. (220.74,104.69) ;
\draw [shift={(221.89,102.33)}, rotate = 115.11] [fill={rgb, 255:red, 0; green, 0; blue, 0 }  ][line width=0.08]  [draw opacity=0] (10.72,-5.15) -- (0,0) -- (10.72,5.15) -- (7.12,0) -- cycle    ;
\draw    (221.89,102.33) .. controls (232.51,140.68) and (260.34,134.99) .. (275.63,105.02) ;
\draw [shift={(276.78,102.67)}, rotate = 115.11] [fill={rgb, 255:red, 0; green, 0; blue, 0 }  ][line width=0.08]  [draw opacity=0] (10.72,-5.15) -- (0,0) -- (10.72,5.15) -- (7.12,0) -- cycle    ;
\draw    (276.78,102.67) .. controls (287.39,141.02) and (315.23,135.32) .. (330.52,105.35) ;
\draw [shift={(331.67,103)}, rotate = 115.11] [fill={rgb, 255:red, 0; green, 0; blue, 0 }  ][line width=0.08]  [draw opacity=0] (10.72,-5.15) -- (0,0) -- (10.72,5.15) -- (7.12,0) -- cycle    ;
\draw    (386.56,103.33) .. controls (397.17,141.68) and (425.01,135.99) .. (440.29,106.02) ;
\draw [shift={(441.44,103.67)}, rotate = 115.11] [fill={rgb, 255:red, 0; green, 0; blue, 0 }  ][line width=0.08]  [draw opacity=0] (10.72,-5.15) -- (0,0) -- (10.72,5.15) -- (7.12,0) -- cycle    ;
\draw    (331.67,103) .. controls (342.89,187.33) and (412.89,186.33) .. (441.44,103.67) ;
\draw    (441.44,103.67) .. controls (475.89,42.33) and (412.89,26.33) .. (386.56,103.33) ;

\draw (129,95.4) node [anchor=north west][inner sep=0.75pt]    {$\cdots $};
\draw (605,82.4) node [anchor=north west][inner sep=0.75pt]  [font=\small]  {$f_{0}$};
\draw (628,71) node [anchor=north west][inner sep=0.75pt]  [font=\small]  {$f_{1}$};
\draw (649,62.4) node [anchor=north west][inner sep=0.75pt]  [font=\small]  {$f_{2}$};
\draw (54,81.4) node [anchor=north west][inner sep=0.75pt]  [font=\small]  {$f_{0}$};
\draw (25,66) node [anchor=north west][inner sep=0.75pt]  [font=\small]  {$f_{1}$};
\draw (2,51.4) node [anchor=north west][inner sep=0.75pt]  [font=\small]  {$f_{2}$};
\draw (188,52.4) node [anchor=north west][inner sep=0.75pt]  [font=\small]  {$f_{1}$};
\draw (241,53.4) node [anchor=north west][inner sep=0.75pt]  [font=\small]  {$f_{2}$};
\draw (423,57) node [anchor=north west][inner sep=0.75pt]  [font=\small]  {$f_{1}$};
\draw (297,52.4) node [anchor=north west][inner sep=0.75pt]  [font=\small]  {$f_{1}$};
\draw (184,132.4) node [anchor=north west][inner sep=0.75pt]  [font=\small]  {$f_{1}$};
\draw (238,132.4) node [anchor=north west][inner sep=0.75pt]  [font=\small]  {$f_{2}$};
\draw (391,130.4) node [anchor=north west][inner sep=0.75pt]  [font=\small]  {$f_{1}$};
\draw (377,166.4) node [anchor=north west][inner sep=0.75pt]  [font=\small]  {$f_{2}$};
\draw (435,30.4) node [anchor=north west][inner sep=0.75pt]  [font=\small]  {$f_{2}$};
\draw (185,79.4) node [anchor=north west][inner sep=0.75pt]  [font=\small]  {$f_{0}$};
\draw (241,81.4) node [anchor=north west][inner sep=0.75pt]  [font=\small]  {$f_{0}$};
\draw (295,81.4) node [anchor=north west][inner sep=0.75pt]  [font=\small]  {$f_{0}$};
\draw (347,81.4) node [anchor=north west][inner sep=0.75pt]  [font=\small]  {$f_{0}$};
\draw (406,81.4) node [anchor=north west][inner sep=0.75pt]  [font=\small]  {$f_{0}$};
\draw (481,81.4) node [anchor=north west][inner sep=0.75pt]  [font=\small]  {$f_{0}$};
\draw (296,133.4) node [anchor=north west][inner sep=0.75pt]  [font=\small]  {$f_{1}$};
\draw (355,53.4) node [anchor=north west][inner sep=0.75pt]  [font=\small]  {$f_{2}$};
\draw (122,114) node [anchor=north west][inner sep=0.75pt]  [font=\small]  {$0$};
\draw (538,114) node [anchor=north west][inner sep=0.75pt]  [font=\small]  {$1$};
\draw (440,110) node [anchor=north west][inner sep=0.75pt]  [font=\small]  {$\frac{1}{2}$};

\end{tikzpicture}

\end{center}
\caption{Nodes represent the points in $X$. The farthest node on the right is 1, the second $\frac{1}{2}$ and so on.}
\label{fig P.T not T.T}
\end{figure}
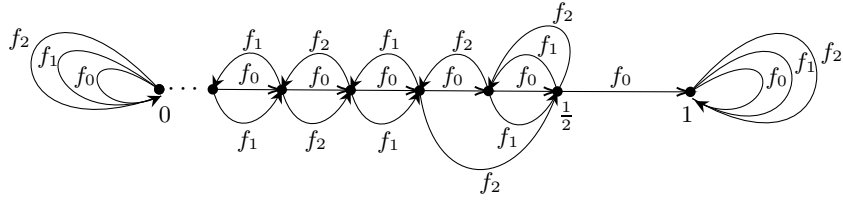
\end{example}
It is worth mentioning  that if $\mathcal{F}$ was homeomorphism in Definition \ref{def intro}, then topological and point transitivity were equivalent \cite{cairns2007topological}.

\subsection{Dynamics along an orbit as a non-autonomous dynamical system}
Let $X$ be a topological space and $f_n:X\to X$ a continuous map for $n\in\N$. 
Then the sequence $\{f_n\}_1^\infty$ denoted by $f_{1,\,\infty}$ defines a non-autonomous discrete dynamical system $(X,\, f_{1,\,\infty})$ \cite{kolyada1996topological}.
In an IFS,
dynamics along a $\sigma$ also defines a non-autonomous   system which we show it by $(X,\,f_\sigma)$ or $f_\sigma:=\{f_{\sigma_i}\}_{i=1}^\infty$ (resp.  $f_\sigma:=\{f_{\sigma_i}\}_{i=-\infty}^{+\infty}$) when 
$\Sigma$ is one sided (resp.  two sided).
If $\Sigma$ is over a finite alphabet, then clearly $f_\sigma$ is defined 
only by finitely many different $f_i$'s.

Let $\sigma=\sigma_1\sigma_2\cdots\in\Sigma$. Then the sequence 
$x,\,f_{\sigma_1}(x),\,f_{\sigma_1\sigma_2}(x),\ldots$ is the \emph{trajectory of $x$ along} $\sigma$ and $ \mathcal{O}^+_{\sigma}(x) $
the set of points in this trajectory is the \emph{(forward) orbit of $x$ along $\sigma$}. 
The backward orbit and backward trajectory may be defined similarly for the case where $\Sigma$ is a two sided subshift. 
Hence one may say that \ifs has property $P$ along  $\sigma\in\Sigma$ if  the respective non-autonomous system $(X,\,f_\sigma)$ has property $P$. By this the following definition may sound abundance, though we bring it for the sake of completeness. 

\begin{defn}\label{defn along an orbit}
	Let \ifs  be an IFS and $U$, $V$ arbitrary non-empty open sets in $X$. Then $\mathfrak{I}$ is called
	\begin{enumerate}
	
		\item 
		\emph{ forward point transitive along an orbit} $\sigma\in\Sigma$, if there is a point $x\in X$, called the transitive point, such that $\overline{\mathcal{O}_\sigma^+(x)}=X$.
	
		\item
		\emph{ topological transitive along an orbit} $\sigma\in\Sigma$, if there is $ n\in\N$ such that $f_{\sigma_1\cdots \sigma_n}(U)\cap V\neq\emptyset$.
		\item 
		\emph{mixing} (resp. \emph{exact}) \emph{along an orbit}  $\sigma\in\Sigma$, if there is $N\in\N$ such that for $n\geq N$, $f_{\sigma_1\cdots\sigma_n}(U)\cap V\neq\emptyset$ (resp. $f_{\sigma_1\cdots\sigma_n}(U)=X$).
	\end{enumerate}
\end{defn}
Similar implications as in \eqref{eq exmixtran} hold here as well. So we have the following result.
\begin{prop}\cite[Proposition 4.6]{sanchez2017chaos}
If an IFS has	topological transitivity along $\sigma$, then it is point transitive along $\sigma$.
\end{prop} 
\begin{proof}
	The proof is similar to the proof of Proposition \ref{prop trans0} by replacing \eqref{eq G_n} with 
	
		$G_n=\cup_{\ell\in\N}(f_{\sigma_1\sigma_2\cdots\sigma_\ell})^{-1}(U_n)$ and applying the same reasoning.
\end{proof}
The converse of the above proposition is not necessarily true as the next example shows. This example also shows that point transitivity along an orbit does not imply that the transitive points are residual along that orbit.
			
\begin{example}\label{Ex trans & top trans}
	Let $X=[0,\,1]$ and
	$ \mathfrak{I}=(X,\,\{f_0,\,f_1\},\,\Sigma_{|\mathcal{F}|}) $ where
	\begin{equation*}
		f_0(x)=
		\begin{cases}
			2x, &  0\leq x \leq\dfrac{1}{2} ,\\
			1 &  \dfrac{1}{2}\leq x \leq 1,
		\end{cases} \quad \text{and}\quad
		f_1(x)=
		\begin{cases}
			0, & 0\leq x \leq\dfrac{1}{2} ,\\
			2x-1, &  \dfrac{1}{2}\leq x \leq 1.
		\end{cases}
	\end{equation*}
	Also let $f:[0,\,1]\to [0,\,1]$ be defined as $f(x)=2x\mod 1$, i.e.
	\begin{equation*}
		f(x)=
		\begin{cases}
			f_0(x), &  0\leq x \leq\dfrac{1}{2} ,\\
			f_1(x), &   \dfrac{1}{2}\leq x \leq 1.
		\end{cases}
	\end{equation*}
	Let $z\in(0,\,1/2)$ be a transitive point of $f$ and set $\sigma:=\sigma_1\sigma_2\cdots\in\Sigma_{|\mathcal{F}|}$ where $\sigma_1=0$ and for $i>1$, $\sigma_i=0$ (resp. $\sigma_i=1$) whenever $f_{\sigma_1\cdots\sigma_{i-1}}(z)\in (0,\,1/2)$ (resp. $f_{\sigma_1\cdots\sigma_{i-1}}(z)\in (1/2,\,1)$).
	By this settings,  $z$ is a transitive point and so  the non-autonomous system $([0,\,1],\,f_\sigma)$ is point transitive, but not topological transitive. Because for $U=(1/2,\,1)$, $V=(0,\,1/2)$ and for any $n\in\N$,  $f_{\sigma_1\cdots \sigma_n}(U)\cap V=\emptyset$.		
\end{example}


\subsection{Transitivity in IFS vs transitivity in the subshift}  \label{sec Transitivity in IFS}
In general there is not a meaningful relation between the dynamical properties of $(\Sigma,\,\tau)$ and that of $\mathfrak{I}$.
For instance consider \eqref{eq IFS} and let $X=[0,\,1]$, $\mathcal{F}=\{f_0(x)=2x\mod 1,\,f_1(x)\equiv 0\}$ and $\Sigma=\Sigma_{|\mathcal{F}|}=\{0,\,1\}^{\N}$. 
Then, for any open set $U\subset X$, there is some $m$ such that for  $u=0^m$, $f_u(U)=X$. 
Thus  interesting dynamics happens along just one point $\sigma\in\Sigma$, i.e., $\sigma=0^\infty$. Surely there will be some relations  when $f_i$'s are surjective and  some conditions exist on $\Sigma$. This is what we are interested to deal with.

So far we know that transitivity along an orbit defined in Definition \ref{defn along an orbit} implies the transitivity of the system defined in Definition \ref{def intro}. However,  the converse is not necessarily true as we will see in  Example \ref{ex trans contrad}.
Having this in mind,  we like to address the following questions  in 
this section. 

\begin{enumerate}
	\item \label{ques equi notions}
	Does transitivities given in \ref{def intro} imply some sort of transitivity given in Definition \ref{defn along an orbit}?
	\item \label{ques large S}
If the answer to the above question is affirmative, 	in which situation there is a transitive $t\in\Sigma$ such that for some $x\in X$, 
	$\overline{{O}^+_{t}(x)}=X$? 
\end{enumerate}


The following example shows, as one expects,  that transitivity depends on the subshift.
\begin{example}\label{ex trans}
	Let	$\mathfrak{I}=(X,\,\mathcal{F},\,\Sigma)$
	where $ \Sigma $ is an SFT generated by $\mathcal{W}=\{01,\,10\} $ and  $ f_0 $ is the shift map on the two sided full shift  $ X=\{0,\,1\}^{\Z} $ and $ f_1=f_0^{-1} $.
	Clearly this system is not point transitive.
	Moreover, if $ W=_{-1}\!\![000]_{1} $ and $ V= _{-1}\!\![111]_{1} $ are two open central cylinders in $ X $  and if $ w $ is any word in $ \Sigma $,
	then $ f^{-1}_{w}W\cap V = \emptyset $ and so $\mathfrak{I}$  is not topological transitive either. However, if $\Sigma$ were generated by  $\mathcal{W}\cup\{0\}$, then the constituent IFS was both topological and point transitive, showing that transitivity depends  on our subshift.
\end{example}


Later the sets $X$ and $\mathcal{F}=\{f_0,\, f_1\}$ introduced in the following example will be used in several occasions, for instance 
in examples \ref{ex trans contrad}, \ref{ex nonSVGL} and \ref{ex not dense}.

\begin{example}\label{Ex in sequence} 
	Let $\{x_n\}_{n\in\Z}$ be an increasing sequence ($x_{n+1}>x_n$) in $[0,\,1]$ such that $\lim_{n\rightarrow +\infty}x_n=1$ and 
	$\lim_{n\rightarrow -\infty}x_n=0$. 
	Let $X$ be the set of points of this sequence together with $0$ and $1$ and equip $X$ with the subset topology of $[0,\,1]$.
	\begin{enumerate}
		\item
		Let $\mathfrak{I}_1=(X,\,\{f_0\},\,\{0^\Z\})$ be the conventional dynamical system where 
		\begin{equation}\label{f_0}
			f_0(x)=
			\begin{cases}
				x &  \text{if}\quad x\in\{0,\,1\},\\
				x_{n+1} &  \text{if}\quad x=x_n.
			\end{cases}
		\end{equation}
		This system is {point transitive} but not {forward point transitive}. In fact, any point in $X\setminus \{0,\,1\}$ is a transitive point  and $0$ and $1$ are fixed points. 
		\item
		Let  $\mathfrak{I}_2=(X,\,\{f_0,\,f_1\},\,\Sigma)$ and let $f_0$ be as in \eqref{f_0} but $f_1$ be defined as 
		\begin{equation}\label{f_1} 
			f_1(x)=
			\begin{cases}
				x  & \text{if}\quad x\in\{0,\,1\},\\
				x_{n-1}  & \text{if}\quad x=x_n.
			\end{cases}
		\end{equation}
		Also let $\Sigma\subseteq \{0,\,1\}^{\N_{0}}$ be generated by $\mathcal{W}=\{w_0,\,w_1,\,w_3,\ldots\}$ 
		so that there are two words in $\mathcal{W}$, say $w_0$ and $w_1$ such that $ |w_0|=|w_1| $ and
		$\frac{0_{w_0}}{|w_0|}=\frac{1_{w_1}}{|w_1|}>\frac{1}{2}$ where $i_{w_j}$ is the number of i's appearing in $w_j$. 
		Then, $\mathfrak{I}_2$ is transitive. 
		To see this, let
		\begin{equation}\label{eq sigma0}
			\sigma^0=w_0w_1w_1w_0w_0w_0w_1w_1w_1w_1\cdots w_0^nw_1^{n+1}w_0^{n+2}w_1^{n+3}\cdots
		\end{equation} 
		and let $x\in X\setminus\{0,\,1\}$. Then, $\overline{\mathcal{O}_{\sigma^0}^+(x)}=X$. One example is when $\Sigma=\Sigma_{|\mathcal{F}|}$ and  $\mathcal{W}=\{w_0=0,\,w_1=1\}$ where then $\frac{0_{w_0}}{|w_0|}=\frac{1_{w_1}}{|w_1|}=1$.
	\end{enumerate}
\end{example}

Now we set up to show that  when $\Sigma$ is an irreducible sofic, functions are semi-open, i.e., interior of image of any open set is non-empty,  and when the respective IFS is topological transitive, then for some  transitive $t\in \Sigma$, one has  point transitivity along $t$. 
This will give an answer to questions \ref{ques equi notions} and \ref{ques large S}  on the beginning of this section for special cases where $\Sigma$ is an irreducible sofic.
First we recall a classical result.

\begin{thm}\label{Boyle Th}
	(Boyle \cite{boyle1983})
	Let $ \Sigma $ and $ \Sigma' $ be irreducible shifts of finite type with $ h(\Sigma)>h(\Sigma') $. Then, there is a factor code from $ \Sigma $ onto $ \Sigma' $ if and only if $ P(\Sigma) \searrow  P(\Sigma')$.
\end{thm}
\begin{prop}\label{prop tra along sigma}
	Let $\mathfrak{I}=(X,\,\mathcal{F},\,\Sigma)$  be a surjective and topological transitive IFS and maps in $\mathcal{F}$ semi-open. 
	 Also let $\Sigma$ be an irreducible  sofic. Then there is a forward transitive $t\in\Sigma$ such that the non-autonomous system $(X,\,f_t)$ is point transitive.
\end{prop}
\begin{proof}
	Let $\mathcal{A}=\{0,\,\ldots,\,k-1\}$ be the set of characters of $\Sigma$. If $\Sigma$ does not have a fixed point,  replace $\mathcal{A}$ with $\mathcal{A}'=\mathcal{A}\cup\{k\}$ and  replace $\mathcal{F}$ with $\{f_0,\ldots,\,f_{k-1}\}\cup\{f_k\}$ where $f_k$ is the identity map and set $\Sigma'$ to be the corresponding subshift whose set of forbidden set is the same as $\Sigma$. Observe that $k^{\N_0}$ is a fixed point of $\Sigma'$ and if $t'\in\Sigma'$ is a transitive point, then $t$ obtained from $t'$ by forgetting the entries whose value is $k$ is transitive in $\Sigma$.
	Thus without loss of generality, we may assume that $\Sigma$ has a fixed point. 
	
	So let $ \mathfrak{I} $ be { topological transitive} and set $ \mathfrak{I}':=(X,\,\mathcal{F},\,\Sigma_{|\mathcal{F}|})$ and let $\mathcal{B}:= \{W_m : m\in \mathbb{N} \} $ be a base for the topology on $ X$. 
	First we construct a transitive point $ t\in \Sigma_{|\mathcal{F}|} $ such that $ \overline{{O}^+_{t}(x)}=X$.
	
	Let $U_m$ be an open set such that $\overline {U_m}\subseteq W_m$. Pick $v_1\in \mathcal{L}(\Sigma_{|\mathcal{F}|})$ such that $f_{v_1}(U_1)\cap U_2\neq\emptyset$ and consider $f_{v_1v'_1}$ where $v'_1$ is the concatenation of all characters or words of length 1. Then by the fact that $f_i$'s are semi-open and our system is topological transitive, there is $v_2$ such that $f_{v_1v'_1v_2}(U_1)\cap U_3\neq\emptyset$. By the same reasoning  and induction argument, there is $v_k$ such that
	for $u_k:=v_1v'_1v_2\cdots v_iv'_iv_{i+1}\cdots v_{k-1}v'_{k-1}v_k$,
	\begin{equation}\label{eq semi-open}
	f_{u_k}(U_1)\cap U_{k+1}\neq\emptyset.
	\end{equation}
	Here $v'_i$ is the concatenation of all words of length $i$.
	 Let  $C_k=\overline{U_1}\cap (f_{u_{k}})^{-1}(\overline{U_{k+1}})$ 
	 be the compact set in $W_1$
	 and notice that $C_{k+1}\subseteq C_k$; in particular,  $\cap_kC_k$ is a nonempty compact set in $W_1$. Thus if $x\in \cap_k C_k$, then $f_{u_k}(x)\in W_k$. This means that our system is point transitive along the transitive $t=v_1v'_1v_2\cdots\in\Sigma_{|\mathcal{F}|}$.
	 So the problem is set when $\Sigma$ is a full shift. 
	
	Now let $ \Sigma $ be SFT and recall that we are assuming that it has a fixed point. This means $P(\Sigma_{|\mathcal{F}|})\searrow P(\Sigma)$ and so by \ref{Boyle Th} there is a factor code $ \phi $ from $ \Sigma_{|\mathcal{F}|} $ onto $ \Sigma $ and in particular there exists a transitive point $\phi( t) \in \Sigma $  with $ \overline{{O}^+_{\phi(t)}(x)}=X $. It remains to prove the case when $ \Sigma $ is sofic. But any sofic is a factor of an SFT and transitivity is preserved by factor codes and take this factor code to be a 1-block factor code. 
	By an argument as above we may extend  this SFT to have a fixed point and the new character, if any, will map to a new added character in character set of $\Sigma$ by the block factor map whose associated map in $\mathfrak{I}$ is identity. 
	As a result, 
	a transitive $ t \in \Sigma $ and $ x\in X $ exists as required.
\end{proof}
In the above proposition, the same conclusion holds if we are sure that  for any $k$, there is $u_k$ such that as in \eqref{eq semi-open}, the intersection has non-empty interior.
In fact we conjecture that this is the case, i.e. if IFS is topological transitive, $\mathcal{F}$ surjective, then for any nonempty open sets $U$ and $V$ there is a $u\in\Sigma_{|\mathcal{F}|}$ such that  $f_u(U)\cap V$ has nonempty interior. In that case we do not require semi-openness in the hypothesis.

 Next we bring examples showing that none of the other conditions on the hypothesis of the above proposition can be ignored.
 \begin{example} The alphabet defining our subshift in the above proposition
 	 was finite; the conclusion is not valid for an infinite case.
	Authors in \cite[Example 2.1]{ghane2019topological} show that in that situation, even when the subshift is a full shift,
	the topological transitivity dose not imply topological transitivity along any orbit. 
 \end{example}
\begin{example}
		Topological transitivity of the IFS in Proposition \ref{prop tra along sigma} cannot be replaced with point transitivity. For instance, the system given in Example \ref{Ex P.T not T.T} had all the conditions on the hypothesis of the proposition (subshift was the full shift, and so sofic and all maps were open) except topological transitivity. There we had point transitivity of the IFS but yet we did not have point transitivity along any orbit.

\end{example}

Now we show that the sofic property cannot be omitted in the hypothesis of the above proposition. Moreover, this example shows that in general  topological  transitivity of an  IFS does not necessarily imply the point transitivity along any $\sigma\in\Sigma$. 
\begin{example}\label{ex trans contrad}
	Let $f_0$, $f_1$ be homeomorphisms defined in Example \ref{Ex in sequence} and let  $\mathfrak{I}=(X,\,\{f_0,\,f_1\},\,\Sigma)$ where $\Sigma\subseteq \{0,\,1\}^{\N}$ is the non-sofic shift generated by $\mathcal{W}=\{0^n1^n:\; n\in\N\}$. 
	Then any $\sigma\in\Sigma$ consists of concatenation of words in $\mathcal{W}$ and their shifts together with points in the closure of them. 
	Thus since $f_{0^n1^n}\equiv \text{id}$ for $n\in\N$, any $\sigma\in\Sigma$ is either concatenation of words in $\mathcal{W}$ or terminating at $0^\infty$ or $1^\infty$. Therefore, $\overline{\mathcal{O}^+_\sigma(x)}\neq X$ for any $\sigma\in\Sigma$ and  $x\in X$.

	On the other hand, any point  $x_0\in X\setminus \{0,\,1\}$ has dense orbit. 
	Because, since $0^{\N_0}$ and $1^{\N_0}$ are points of $\Sigma$,  $x_0$ can travel left and right as far as required by $f_0$ and $f_1$ respectively.
	As a result 		$\mathfrak{I}$ is topological transitive and yet  not point transitive along any $ \sigma\in \Sigma $.
\end{example}

\section{The abundance of  point transitive non-autonomous systems in an IFS}\label{subs st of S}
When a dynamical property such as transitivity, mixing, exactness occur along a $\sigma\in\Sigma$, then the IFS will posses that property as well; though the converse is not necessarily true. In fact, it may not even hold along just a single orbit. In this section we investigate transitivity in this respect.

Let 
\begin{equation}\label{eq S}
	S=	S(\mathfrak{I}):=\{\sigma\in\Sigma:\; \exists\, x\in X\text{ s.t. } \overline{\mathcal{O}_\sigma^+(x)}=X\}.
\end{equation}
 In general, except in few cases, a definite structure cannot be given for $S$, though its largeness can be understood in some cases. Let us demonstrate how different $S$ can be.
\begin{example}\label{ex 2in1}
	\begin{enumerate}
		\item
	$S$ may be all of $\Sigma$. 	 For an example,
	 let  $ f_0(x) =2x \mod 1 $ and $ f_1(x)=3x \mod 1  $ and consider
	 $\mathfrak{I}=([0,\,1],\,\{f_0,\,f_1\},\,\Sigma_{|\mathcal{F}|}) $.
	
		\item 
		$S$ may be an empty set.  This is the case when we have  an IFS which is not point transitive. Though even for a {topological transitive IFS,} $S$ still may be empty (see Example \ref{ex trans contrad}). 

		\item
		$S$ may be residual and yet not all of $\Sigma$. The IFS in Example \ref{Ex trans & top trans} has such property.

		\item
		$S$ may be dense and uncountable, yet not a residual subset. See Example \ref{ex not dense}.
	\end{enumerate}
	
\end{example}

Now we give sufficient conditions for $S$ being dense in $\Sigma$.
First a weaker version of specification property for subshifts:

\begin{defn}
	A subshift $\Sigma$ is called  a \emph{subshift of variable gap length} or SVGL, if there exists $M\in\N$ such that for $u$ and $v$ in $\mathcal{L}(\Sigma)$, there is $w$ with  $|w|\leq M$ and $uwv\in \mathcal{L}(\Sigma)$.
\end{defn}
When $\Sigma$ is mixing and SVGL, then  $\Sigma$ has  \emph{specification property} and in this situation, there exists  $M\in\N$ such that for $u$ and $v$ in $\mathcal{L}(\Sigma)$ there is $w$ with  $|w|= M$ and $uwv\in \mathcal{L}(\Sigma)$.
Clearly an SVGL is irreducible. Moreover, all  sofics are SVGL; however, there are SVGL's which are not sofic.
The SVGL is called almost specification property in
  \cite{jung2011existence}.
\begin{prop}\label{prop SVGL}
	Let $\mathfrak{I}=(X,\,\mathcal{F},\,\Sigma)$  be point transitive along some $\sigma\in\Sigma$,   $\mathcal{F}$ surjective and $\Sigma$ an SVGL. 
	Then, $S$ defined in \eqref{eq S} is dense in $\Sigma$. If $S\neq\Sigma$, then $\Sigma\setminus S$ is also dense in $\Sigma$.
\end{prop}
\begin{proof}
	We prove the first part; the other part has similar proof.
	
	Choose any $\sigma=\sigma_1\sigma_2\cdots\in\Sigma$ such that $\overline{\mathcal{O}_\sigma^+(x)}=X$. Let $[u]$ be a cylinder in $\Sigma$ and use the SVGL property of $\Sigma$ to pick $w_n\in\mathcal{L}(\Sigma)$ such that $uw_n\sigma_1\sigma_2\cdots\sigma_n\in\mathcal{L}(\Sigma)$ with $|w_n|\leq M$ where $M$ is provided by the definition of SVGL.
	Since $\{w_n\in\mathcal{L}(\Sigma):\; |w_n|\leq M\}$ is finite, 
	there is a $w$ and an infinite  subsequence $n_i$ such that for all $i$, $w_{n_i}=w$. Let  $\sigma'=uw\sigma_1\sigma_2\cdots$ be the unique point in $\cap_{i\in\N}[uw\sigma_1\sigma_2\cdots\sigma_{n_i}]$ and observe that for $x'\in  f^{-1}_{uw}(x)$, $\overline{\mathcal{O}_{\sigma'}^+(x')}=X$. This implies $\sigma'\in [u]\cap S$ and since $[u]$ was arbitrary, we are done.
\end{proof}

\begin{rem}
	Assume the hypothesis of Proposition \ref{prop SVGL} and
	let for some $\sigma\in\Sigma$, $\omega_\sigma(x)$ be the \emph{$\omega$ limit set of $x$ along $\sigma$}, that is the limit set of $\mathcal{O}^+_\sigma(x)=\{f_{\sigma_1\sigma_2\cdots\sigma_n}(x):\;n\in\N\}$. The proof of Proposition \ref{prop SVGL} shows that
	$$\{\sigma'\in\Sigma:\;\exists\,x'\in X\text{ s.t. }\omega_{\sigma'}(x')=\omega_\sigma(x)\}$$ is dense in $\Sigma$.
\end{rem}

Now we show that the SVGL property is a necessity in the hypothesis of Proposition \ref{prop SVGL}.
\begin{example}\label{ex nonSVGL}
	Let $X$, $ w_0=010 $, $ w_1=101 $, $f_0$ and $f_1$ be as in Example \ref{Ex in sequence}  and  set $u_0=000$. 
	Let $\sigma^0=\sigma_{1}^0\sigma_{2}^0\cdots \sigma_{\ell}^0\cdots=w_0w_1\cdots$ be defined as in \eqref{eq sigma0} and let $\mathfrak{I}=(X,\,\{f_0,\,f_1\},\, \Sigma)$ where $\Sigma$ is generated by 
	\[
	\mathcal{W}=\{u_0w_0^\ell\sigma_{1}^0\sigma_{2}^0\cdots \sigma_{\ell}^0w_0^\ell u_0:\; \ell\in\N\setminus\}.
	\] 
	See a presentation for $\Sigma$ in Figure \ref{fig coversyn}.
	\begin{figure}
		\begin{center}
			
			\tikzset{every picture/.style={line width=0.45pt}} 
			
			
			\begin{tikzpicture}[x=0.65pt,y=0.40pt,yscale=-1,xscale=1]
				
				\draw    (312.02,410.88) -- (250.47,363.54) ;
				\draw [shift={(248.0,359.59)}, rotate = 394] [color={rgb, 255:red, 0; green, 0; blue, 0 }  ][line width=0.75]    (10.93,-3.29) .. controls (6.95,-1.4) and (3.31,-0.3) .. (0,0) .. controls (3.31,0.3) and (6.95,1.4) .. (10.93,3.29)   ;
				\draw [shift={(313.0,410.5)}, rotate = 217.57] [color={rgb, 255:red, 0; green, 0; blue, 0 }  ][line width=0.75]        ;
				\draw    (381.96,364.66) -- (315.53,411.17) ;
				\draw [shift={(313.0,415.5)}, rotate = 327] [color={rgb, 255:red, 0; green, 0; blue, 0 }  ][line width=0.75]    (10.93,-3.29) .. controls (6.95,-1.4) and (3.31,-0.3) .. (0,0) .. controls (3.31,0.3) and (6.95,1.4) .. (10.93,3.29)   ;
				\draw [shift={(383.89,363.32)}, rotate = 145.01] [color={rgb, 255:red, 0; green, 0; blue, 0 }  ][line width=0.75]         ;
				\draw    (248.89,359.97) -- (248.89,299.32) ;
				\draw [shift={(248.89,297.32)}, rotate = 450] [color={rgb, 255:red, 0; green, 0; blue, 0 }  ][line width=0.75]    (10.93,-3.29) .. controls (6.95,-1.4) and (3.31,-0.3) .. (0,0) .. controls (3.31,0.3) and (6.95,1.4) .. (10.93,3.29)   ;
				\draw [shift={(248.89,362.32)}, rotate = 270] [color={rgb, 255:red, 0; green, 0; blue, 0 }  ][line width=0.75]        ;
				\draw    (383.89,361.32) -- (383.89,298.32) ;
				\draw [shift={(383.89,363.32)}, rotate = 270] [color={rgb, 255:red, 0; green, 0; blue, 0 }  ][line width=0.75]    (10.93,-3.29) .. controls (6.95,-1.4) and (3.31,-0.3) .. (0,0) .. controls (3.31,0.3) and (6.95,1.4) .. (10.93,3.29)   ;
				\draw    (251.24,297.33) -- (381.89,298.3) ;
				\draw [shift={(383.89,298.32)}, rotate = 180.42] [color={rgb, 255:red, 0; green, 0; blue, 0 }  ][line width=0.75]    (10.93,-3.29) .. controls (6.95,-1.4) and (3.31,-0.3) .. (0,0) .. controls (3.31,0.3) and (6.95,1.4) .. (10.93,3.29)   ;
				\draw [shift={(248.89,297.32)}, rotate = 0.42] [color={rgb, 255:red, 0; green, 0; blue, 0 }  ][line width=0.75]      ;
				\draw    (383.89,296.32) -- (383.89,235.67) ;
				\draw [shift={(383.89,233.32)}, rotate = 270] [color={rgb, 255:red, 0; green, 0; blue, 0 }  ][line width=0.75]         ;
				\draw [shift={(383.89,298.32)}, rotate = 270] [color={rgb, 255:red, 0; green, 0; blue, 0 }  ][line width=0.75]    (10.93,-3.29) .. controls (6.95,-1.4) and (3.31,-0.3) .. (0,0) .. controls (3.31,0.3) and (6.95,1.4) .. (10.93,3.29)   ;
				\draw    (248.89,294.97) -- (248.89,234.32) ;
				\draw [shift={(248.89,232.32)}, rotate = 450] [color={rgb, 255:red, 0; green, 0; blue, 0 }  ][line width=0.75]    (10.93,-3.29) .. controls (6.95,-1.4) and (3.31,-0.3) .. (0,0) .. controls (3.31,0.3) and (6.95,1.4) .. (10.93,3.29)   ;
				\draw [shift={(248.89,297.32)}, rotate = 270] [color={rgb, 255:red, 0; green, 0; blue, 0 }  ][line width=0.75]      ;
				\draw    (251.24,232.35) -- (312.89,233.29) ;
				\draw [shift={(314.89,233.32)}, rotate = 180.87] [color={rgb, 255:red, 0; green, 0; blue, 0 }  ][line width=0.75]    (10.93,-3.29) .. controls (6.95,-1.4) and (3.31,-0.3) .. (0,0) .. controls (3.31,0.3) and (6.95,1.4) .. (10.93,3.29)   ;
				\draw [shift={(248.89,232.32)}, rotate = 0.87] [color={rgb, 255:red, 0; green, 0; blue, 0 }  ][line width=0.75]        ;
				\draw    (317.24,233.35) -- (378.89,234.29) ;
				\draw [shift={(380.89,234.32)}, rotate = 180.87] [color={rgb, 255:red, 0; green, 0; blue, 0 }  ][line width=0.75]    (10.93,-3.29) .. controls (6.95,-1.4) and (3.31,-0.3) .. (0,0) .. controls (3.31,0.3) and (6.95,1.4) .. (10.93,3.29)   ;
				\draw [shift={(314.89,233.32)}, rotate = 0.87] [color={rgb, 255:red, 0; green, 0; blue, 0 }  ][line width=0.75]       ;
				\draw    (248.89,229.97) -- (248.89,182.32) ;
				\draw [shift={(248.89,180.32)}, rotate = 450] [color={rgb, 255:red, 0; green, 0; blue, 0 }  ][line width=0.75]    (10.93,-3.29) .. controls (6.95,-1.4) and (3.31,-0.3) .. (0,0) .. controls (3.31,0.3) and (6.95,1.4) .. (10.93,3.29)   ;
				\draw [shift={(248.89,232.32)}, rotate = 270] [color={rgb, 255:red, 0; green, 0; blue, 0 }  ][line width=0.75]      ;
				\draw    (383.89,231.32) -- (383.89,183.67) ;
				\draw [shift={(383.89,181.32)}, rotate = 270] [color={rgb, 255:red, 0; green, 0; blue, 0 }  ][line width=0.75]       ;
				\draw [shift={(383.89,233.32)}, rotate = 270] [color={rgb, 255:red, 0; green, 0; blue, 0 }  ][line width=0.75]    (10.93,-3.29) .. controls (6.95,-1.4) and (3.31,-0.3) .. (0,0) .. controls (3.31,0.3) and (6.95,1.4) .. (10.93,3.29)   ;
				\draw    (249.89,136.97) -- (249.89,89.32) ;
				\draw [shift={(249.89,87.32)}, rotate = 450] [color={rgb, 255:red, 0; green, 0; blue, 0 }  ][line width=0.75]    (10.93,-3.29) .. controls (6.95,-1.4) and (3.31,-0.3) .. (0,0) .. controls (3.31,0.3) and (6.95,1.4) .. (10.93,3.29)   ;
				\draw [shift={(249.89,139.32)}, rotate = 270] [color={rgb, 255:red, 0; green, 0; blue, 0 }  ][line width=0.75]      ;
				\draw    (382.89,138.32) -- (382.89,90.67) ;
				\draw [shift={(382.89,88.32)}, rotate = 270] [color={rgb, 255:red, 0; green, 0; blue, 0 }  ][line width=0.75]     ;
				\draw [shift={(382.89,140.32)}, rotate = 270] [color={rgb, 255:red, 0; green, 0; blue, 0 }  ][line width=0.75]    (10.93,-3.29) .. controls (6.95,-1.4) and (3.31,-0.3) .. (0,0) .. controls (3.31,0.3) and (6.95,1.4) .. (10.93,3.29)   ;
				\draw    (249.89,84.97) -- (249.89,51.32) ;
				\draw [shift={(249.89,49.32)}, rotate = 450] [color={rgb, 255:red, 0; green, 0; blue, 0 }  ][line width=0.75]    (10.93,-3.29) .. controls (6.95,-1.4) and (3.31,-0.3) .. (0,0) .. controls (3.31,0.3) and (6.95,1.4) .. (10.93,3.29)   ;
				\draw [shift={(249.89,87.32)}, rotate = 270] [color={rgb, 255:red, 0; green, 0; blue, 0 }  ][line width=0.75]      ;
				\draw    (382.89,86.32) -- (382.89,53.67) ;
				\draw [shift={(382.89,51.32)}, rotate = 270] [color={rgb, 255:red, 0; green, 0; blue, 0 }  ][line width=0.75]      ;
				\draw [shift={(382.89,88.32)}, rotate = 270] [color={rgb, 255:red, 0; green, 0; blue, 0 }  ][line width=0.75]    (10.93,-3.29) .. controls (6.95,-1.4) and (3.31,-0.3) .. (0,0) .. controls (3.31,0.3) and (6.95,1.4) .. (10.93,3.29)   ;
				\draw    (252.24,87.32) -- (295.89,87.32) ;
				\draw [shift={(297.89,87.32)}, rotate = 180] [color={rgb, 255:red, 0; green, 0; blue, 0 }  ][line width=0.75]    (10.93,-3.29) .. controls (6.95,-1.4) and (3.31,-0.3) .. (0,0) .. controls (3.31,0.3) and (6.95,1.4) .. (10.93,3.29)   ;
				\draw [shift={(249.89,87.32)}, rotate = 0] [color={rgb, 255:red, 0; green, 0; blue, 0 }  ][line width=0.75]      ;
				\draw    (337.24,88.32) -- (380.89,88.32) ;
				\draw [shift={(382.89,88.32)}, rotate = 180] [color={rgb, 255:red, 0; green, 0; blue, 0 }  ][line width=0.75]    (10.93,-3.29) .. controls (6.95,-1.4) and (3.31,-0.3) .. (0,0) .. controls (3.31,0.3) and (6.95,1.4) .. (10.93,3.29)   ;
				\draw [shift={(334.89,88.32)}, rotate = 0] [color={rgb, 255:red, 0; green, 0; blue, 0 }  ][line width=0.75]     ;
				
				\draw (257,390.4) node [anchor=north west][inner sep=0.75pt]    {$u_{0}$};
				\draw (350.89,391.22) node [anchor=north west][inner sep=0.75pt]    {$u_{0}$};
				\draw (310.89,420.22) node [anchor=north west][inner sep=0.75pt]    {$b$};
				\draw (226,322.4) node [anchor=north west][inner sep=0.75pt]    {$w_{0}$};
				\draw (388,322.4) node [anchor=north west][inner sep=0.75pt]    {$w_{0}$};
				\draw (226,254.4) node [anchor=north west][inner sep=0.75pt]    {$w_{0}$};
				\draw (388,254.4) node [anchor=north west][inner sep=0.75pt]    {$w_{0}$};
				\draw (300,265) node [anchor=north west][inner sep=0.75pt]    {$\sigma^0 _{{1}}$};
				\draw (265,201) node [anchor=north west][inner sep=0.75pt]    {$\sigma^0 _{{1}}$};
				\draw (335,201) node [anchor=north west][inner sep=0.75pt]    {$\sigma^0 _{{2}}$};
				\draw (303.5,82.5) node [anchor=north west][inner sep=0.75pt]    {$\cdots $};
				\draw (226,202.4) node [anchor=north west][inner sep=0.75pt]    {$w_{0}$};
				\draw (388,202.4) node [anchor=north west][inner sep=0.75pt]    {$w_{0}$};
				\draw (226,105.4) node [anchor=north west][inner sep=0.75pt]    {$w_{0}$};
				\draw (388,105.4) node [anchor=north west][inner sep=0.75pt]    {$w_{0}$};
				\draw (226,58.4) node [anchor=north west][inner sep=0.75pt]    {$w_{0}$};
				\draw (388,58.4) node [anchor=north west][inner sep=0.75pt]    {$w_{0}$};
				\draw (246,1.4) node [anchor=north west][inner sep=0.75pt]    {$\vdots $};
				\draw (380,1.4) node [anchor=north west][inner sep=0.75pt]    {$\vdots $};
				\draw (246,131) node [anchor=north west][inner sep=0.75pt]    {$\vdots $};
				\draw (380,131) node [anchor=north west][inner sep=0.75pt]    {$\vdots $};
				\draw (263,57) node [anchor=north west][inner sep=0.75pt]    {$\sigma^0 _{{1}}$};
				\draw (344,57) node [anchor=north west][inner sep=0.75pt]    {$\sigma^0 _{{n}}$};

			\end{tikzpicture}

		\end{center}
		\caption{ Any word in $\mathcal{W}$ starts and terminates at $b$.}
		\label{fig coversyn}
	\end{figure}
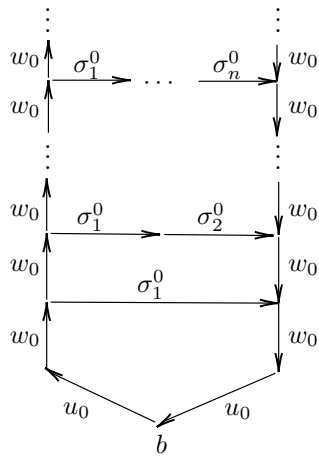

	If $v=v_1\cdots v_{|v|}\in\mathcal{W}$,  then 0 and 1 are fixed by $f_v$ and for any other  $x_i\in X\setminus\{0,\,1\}$, $f_{v_1\cdots v_\ell}(x_i)=x_j$ where $1\leq \ell\leq |v|$ and $j >  i$.
	
	Observe that $\sigma^0\in\Sigma$  does not have $u_0$ as a subword  and also by the same reasoning for $\mathfrak{I}_2$ in Example  \ref{Ex in sequence} , $\overline{{O}^+_{\sigma^0}(x)}=X$ for $x\in X\setminus\{0,\,1\}$.
	On the other hand if $u_0$ appears in a $ \sigma\in\Sigma$ infinitely (resp. finitely) many times, then by our construction where any $u_0$ appears only on the beginning or ending of members of  $\mathcal{W}$, 
	this $\sigma$ must start with a terminal subword of a $w\in \mathcal{W}$, may be empty, and afterwards has some infinite concatenation of the members of $\mathcal{W}$ (resp. eventually will terminate at $w_0^\infty$). 
	This in turn implies that $\overline{{O}^+_{\sigma}(x)}\neq X$ for any $x\in X$. 
	In fact then 0 and 1 are fixed by the orbit along $\sigma$ and any other $x$ marches to 1 along that orbit with some relatively minor fluctuations.
	Hence, if $ \sigma $ is transitive then for any $ x,\; \overline{{O}^+_{\sigma}(x)}\neq X $.
	In particular, 
	if $\sigma\in[u_0]$, then $\sigma\not\in S$ and consequently $S$ is not dense in $\Sigma$. 
\end{example}

Since $u_0u_0$ is a synchronizing word,
the above example is synchronized. For irreducible shifts we have the following implications.
$$\text{ full shift }\Rightarrow\text{ SFT }\Rightarrow\text{ sofic }\Rightarrow\text{ SVGL }\Rightarrow\text{ synchronized }\Rightarrow\text{ coded}.$$
Thus the transitive non-autonomous systems in an IFS whose subshift is synchronized or beyond may be scarce.

Observe that by Proposition \ref{prop tra along sigma}, the conclusion of Proposition \ref{prop SVGL} is immediate when $\Sigma$ is sofic; that is
because the orbit of a transitive $\sigma$, attained by the above proposition,  is again in $S$ and is dense in $\Sigma$. 
However, still we cannot
guarantee that $S$ is residual as the next example shows,  even for a case where $\Sigma$ is a mixing SFT. 

\begin{example}\label{ex not dense}
	Let $X$ and $\mathcal{F}=\{f_0,\, f_1\}$ be as in Example \ref{Ex in sequence} and 
	consider $\mathfrak{I}_\Sigma=(X,\,\mathcal{F},\,\Sigma)$. 
	\begin{enumerate}
		\item
		First let $\Sigma=\Sigma_{|\mathcal{F}|}$ and let  $ w_i $  be a word consisting  of the concatenation of all words of length $ i\in\N $ in $\mathcal{L}(\Sigma_{|\mathcal{F}|})$ and notice that
		$\frac{0_{w_i}}{|w_i|}=\frac{1}{2}$. As a result, if $u={1^{|w_i|}{w_i}}$, then
		$f_u(x)$ moves $x\not\in\{0,\,1\}$ at least $\frac{|w_i|}{2}$ to left. Therefore, for the transitive
		$$ t=1^{|w_1|}w_11^{|w_2|}w_21^{|w_3|}w_3 \cdots\in\Sigma_{|\mathcal{F}|},$$
		$ 1\notin\overline{\mathcal{O}^+_{t}(x)} $ and so $\overline{\mathcal{O}^+_{t}(x)}\neq X$. In particular, this shows that the conclusion of Proposition \ref{prop tra along sigma} is not necessarily valid for all transitive points in an irreducible sofic shift. Clearly $S(\mathfrak{I}_{\Sigma_{|\mathcal{F}|}})$, although dense, it is not closed and hence it is not a subshift.

		\item 
		To complete our collection of the possible various cases of $S(\mathfrak{I})$, we construct an example where $S(\mathfrak{I})$ is a dense uncountable but not residual subset of the subshift. To do this
		let $\Sigma_\mathcal{W}$ be the SFT generated by $\mathcal{W}=\{w_0=100,\,w_1=011,\, w_2=000\}$ and call the associated IFS $\mathfrak{I}_{\mathcal{W}}$. 
		
		We have 	$\frac{0_{w_0}}{|w_0|}=\frac{1_{w_1}}{|w_1|}=\frac{2}{3}$. Hence if $\sigma^0$ is chosen as in \eqref{eq sigma0}, then $\overline{\mathcal{O}^+_{\sigma^0}(x)}=X$ 
		for $x\in X\setminus \{0,\,1\}$. However, $\mathcal{O}^+_{0^\infty}(x)$ is not dense for any $x\in X$ and hence $S(\mathfrak{I}_\mathcal{W})$ is not closed and again not a subshift. Also,  observe that  the subshift $\Sigma_{\mathcal{W}'}$ generated by 	$\mathcal{W}'=\cup_{k\in\N}\{w_0^kw_1^k,\,w_1^kw_0^k\}$	 is a subsystem of $\Sigma_\mathcal{W}$ and any transitive point of that lies in $S(\mathfrak{I}_\mathcal{W})$. The latter follows from the fact that $f_{u}(x)=x$ for $u=w_0^kw_1^k$ or $u=w_1^kw_0^k$
		and the fact that $w_0^kw_1^kw_1^kw_0^k$ is a subword for a transitive point in $\Sigma_{\mathcal{W}'}$ for any $k\in\N$.
		Thus any $x\in X\setminus\{0,\,1\}$ moves left and right as far as possible.
		This implies that $S(\mathfrak{I}_{\mathcal{W}'})\subset S(\mathfrak{I}_\mathcal{W})$ has uncountable points.

		Now we show that in this example $S(\mathfrak{I}_\mathcal{W})$ is  not a residual subset of $\Sigma_\mathcal{W}$. 
		It is an easy consequence of the Birkhoff's ergodic theorem that the frequency of $w_i\in\mathcal{W}$ is $\frac{1}{3}$ for almost all $\sigma\in\Sigma_\mathcal{W}$ 
		(we are considering the Markov measure on $\Sigma_\mathcal{W}$: A unique ergodic Borel measure $\mu$ which is positive on open sets and has the maximum metric entropy among all other measures). 
		This means that occurrence of $0$ is as twice as that of $1$ for almost all $\sigma$.
		Thus for   $x\in X$ and almost all $\sigma$, $\overline{{O}^+_{\sigma}(x)}\neq X$. This in turn implies that $\mu(S(\mathfrak{I}_\mathcal{W}))=0$. Now if $S(\mathfrak{I}_\mathcal{W})$  was residual in $\Sigma_\mathcal{W}$, then $S(\mathfrak{I}_\mathcal{W})$ would be measurable and since it is shift invariant it must have full measure which is impossible for this example.
		
		If one chooses $w_2$ in $\mathcal{W}$
		to be $0000$, then gcd$\{|w_i|\,|\; w_i\in\mathcal{W},\, 0\leq i\leq 2\}=1$ which implies that $\Sigma_\mathcal{W}$ is a mixing SFT (\cite{dastjerdi2019mixing, epperlein2019mixing}). So either mixing or non-mixing, there are examples that $S$, in spite of being invariant and having a  transitive point under the shift map, is not residual.
	\end{enumerate}
\end{example}  
\section{Mixing and exactness in an IFS} \label{sec Mixing in IFS }
Clearly mixing along an orbit given in Definition \ref{defn along an orbit} implies mixing defined in Definition \ref{def intro} and the converse is not true as the next example shows.
This example also shows that 
if the IFS is mixing, unlike Proposition \ref{prop tra along sigma}, 
one cannot have mixing along an orbit.

\begin{example}\label{Ex 0x1x}
	Let $ \mathfrak{I}=\left(X=\{0,\,1\}^\mathbb{N},\,\mathcal{F}=\{f_0,\,f_1\},\,\Sigma_{|\mathcal{F}|}\right) $
	and for $\xi=\xi_1\xi_2\cdots\in X$ define 
	\begin{align*}
		f_0(\xi)&=0\xi=0\xi_1\xi_2\cdots,\\
		f_1(\xi)&=1\xi=1\xi_1\xi_2\cdots.
	\end{align*}
	For $w=w_1\cdots w_{n-1}w_{n}$, set $w^{-1}:=w_nw_{n-1}\cdots w_1$ and observe that
	$f_w(\xi)=w^{-1}\xi$.
	Now let $[u]$ and $[v]$ be any cylinder and set $M:=|v|$. Then for $m\geq M$ and $ w\in \mathcal{L}_{m}(\Sigma) $, $f_w([u])\cap [v]\neq\emptyset$ iff
	$w$ is a word terminating at $v^{-1}$  and hence $ \mathfrak{I} $  is mixing. 
	On the other hand assume $ \sigma\in \Sigma ,\, v=100 $ and $ u $ any word.
	Now for $m\geq 2$, if $ f_{\sigma_1\ldots \sigma_m}([u])\cap [v] \neq \emptyset $, then $ w=\sigma_1\cdots \sigma_m $ terminates at $ v^{-1} $ but neither $ w0 $ nor $ w1 $ terminates at $ v^{-1} $. This  implies that both $ f_{w0}([u])\cap [v]$ and $f_{w1}([u])\cap [v] $ are empty sets. Thus $ \mathfrak{I} $ is not  mixing along any orbit $ \sigma $. 
\end{example}
Next example shows that simple dynamics in the individual maps in an IFS may raise to a rich dynamics in the IFS.
Intuitively, if we have two maps in an IFS that one flows all the point in a definite direction and the other on the opposite direction, then the arbitrary combination of these maps can create a complicate dynamics. Example \ref{Ex 0x1x} had this property but the IFS was not as rich as the following.
\begin{example}\label{ex 2 non exact}
	Here we give an example such that none of the maps of the IFS, considering as a conventional dynamical system, is transitive but the IFS itself is exact and thus mixing and {topological transitive.}
	
	Let $ \mathfrak{I}=(X=\{0,\,1\}^{\N},\,\{f_0,\,f_1\},\, \Sigma_{|\mathcal{F}|}=\{0,\,1\}^{\N}) $ be an IFS where for $ \zeta=\zeta_1\zeta_2\cdots \in \{0,\,1\}^{\N} $,
	\begin{equation}\label{eq f_i}
		f_i(\zeta)=
		\begin{cases}
			i\zeta_1\zeta_2\cdots & \text{if}\quad  \zeta_1=i,\\
			\zeta_2\zeta_3\cdots & \text{if}\quad  \zeta_1\neq i .
		\end{cases}
	\end{equation}
	we have the following observations
	\begin{enumerate}
		\item 
		$f_i$ is a finite to 1 surjective open map 
		with 
		$0^\infty$ and $1^\infty$ its only fixed points.
		\item
		$f_0$ (resp. $f_1$) attracts all points in $X\setminus\{1^\infty\}$ (resp. $X\setminus\{0^\infty\}$) and leaves the point $1^\infty$ (resp. $0^\infty$) fixed. Thus $f_i$ is not transitive and has a very simple dynamics.
		\item
		Any $\zeta=\zeta_1\zeta_2\cdots\in X$ is periodic of any given even period  $p=2q\in\N$ along a $\sigma\in\Sigma$. 
		To see this  set  
		$$ \sigma=(\zeta_1^q{\zeta_1^*}^q)^\infty=\left(\overbrace{\zeta_1\zeta_1\cdots \zeta_{1}}^{q \text{ times}}\overbrace{\zeta_1^*\zeta_1^*\cdots \zeta_{1}^*}^{q\text{ times}}\right)^\infty $$ 
		where for $a\in\mathcal{A}=\{0,\,1\}$,
		\begin{equation}\label{eq a*}
			a^*=
			\begin{cases}
				1,\quad a=0\\
				0,\quad a=1.
			\end{cases}
		\end{equation}
		
		Also, any transitive $\zeta=\zeta_1\zeta_2\cdots\in X$  is transitive along the transitive point $\zeta^*=\zeta_1^*\zeta_2^*\cdots\in\Sigma$. A point such as $\zeta=(\zeta_1\zeta_2\cdots\zeta_{p})^\infty\in X$ is periodic of period $p$ along the periodic point $\zeta^*=(\zeta_1^*\zeta_1^*\cdots\zeta_{p}^*)^\infty\in \Sigma$.  
		\item
		$\mathfrak{I}$ is exact along a transitive point.
		\begin{proof}
			Fix an  open set $ U\subseteq X $ and pick $w\in\mathcal{L}_{k}(\Sigma_{|\mathcal{F}|})$ such that ${[w]}\subseteq U$.  Set $ w^*:=w_0^*\cdots w_{k}^* $, $w^*_i$ defined as in \eqref{eq a*}, and note that  $X= f_{w^*v}([w])$ where $w^*v$ is any word whose initial segment is $w^*$.
			
			The set $\mathcal{L}_m(\Sigma_{|\mathcal{F}|})$ has $2^m$ words.
			Set $ \mathcal{P}_m(\mathcal{L}_m(\Sigma_{|\mathcal{F}|}))=\{v_1^m,\,\ldots,\,v_{2^m!}^m\}\subseteq \mathcal{L}_{m2^m}(\Sigma_{|\mathcal{F}|}) $ to be the set of $2^m!$ words constructed from the permutation of  words in $\mathcal{L}_m(\Sigma_{|\mathcal{F}|})$ and for $n>m$, let
			$$t=v_1^1v_2^1\cdots  v_1^m\cdots v_{2^m!}^m\cdots  v_1^n\cdots v_{2^n!}^n\cdots=u_1u_2\cdots\in \Sigma_{|\mathcal{F}|},$$
			be the transitive point where $u_1=v_1^1$, $u_2=v_2^1$ and so on. So each $u_i$ is one of the $v_j^m$'s coming after each other in the obvious order.
			Observe that $u_i$ has the same number of 0's and 1's and any word $v\in\mathcal{L}(\Sigma_{|\mathcal{F}|})$ appears as the initial segment of infinitely many $u_i$'s. 	We will show that $\mathfrak{I}$ is exact along $t$.

			Another observation is that  for any word $b$ such as $u_i$ whose 0's and 1's are equal, and any cylinder $[a]$, $|f_b([a])|\leq |[a]|$.
			
			Set $$[a_i]:=f_{u_1\cdots u_i}([w])$$
			and note that $\{|a_i|\}_{i\in\N}$ is a non-increasing sequence.
			Moreover, if $|a_{i+1}|<|a_i|$ for some $|w|$ instances of $i$'s along $t$, then call the last instance $\ell$ and notice that then $f_{u_1\cdots u_\ell}([w])=X$ and so in this case this IFS is exact along $t$. 
			Otherwise, without loss of generality assume that for all $i\in\N$, $|a_i|=|w|$.
			We will show that this latter case does not happen and so we are done.
			
			First let $b=b_1\cdots b_{|b|}$ be any word  and let $|f_b([a])|= |[a]|$ where $a=a_1\cdots a_{|a|}$.
			Let $m(a,\,b)=\min\{|f_{b_1\cdots b_i}([a])|:\; 1\leq i\leq |b|\}$ and 
			set
			$$\alpha=\alpha(a,\,b):=\max\{
			i:\; |f_{b_1\cdots b_i}([a])|=m(a,\,b),\, 1\leq i\leq |b|\}.$$
			In other words, $\alpha(a,\,b)$ is the last instance where $f_{b_1\cdots b_i}([a])$ has the shortest length.
			Let  
			$f_{b_1\cdots b_\alpha}([a])=[a']=[a'_1\cdots a'_{|a'|}]$ for some $a'$, $|a'|<|a|$. In fact $a'$ is the terminal segment of $a$.
			Since $|f_{b_1\cdots b_{\alpha+i}}([a])|>|a'|$ for $1\leq i\leq |b|-\alpha$,  by the definition of $f_j$'s, $b_{\alpha+1}=a'_1$ and in particular
			$f_b([a])=[a'^{\beta(a,\,b)}_1a']$ where $$\beta(a,\,b)=|a|-|a'|.$$ 
			
			Now assume $|a_i|=|w|$ and set 
			$\alpha_{i}=\alpha(a_i,\,u_{i+1})$ and 
			$\beta_i=\beta(a_i,\,u_{i+1})$.
			If $\beta_{i+1}\leq \beta_{i}$, then 
			$[a_i]=[a_{i+1}]$. 
			So if there is $M\in\N$ such that for $i\geq M$, $\beta_{i+1}\leq \beta_{i}$; or equivalently, for $i\geq M$, $[a_i]=[a_{M}]$, 
			then along $t$ we arrive at a $u_{\ell}$ whose initial segment is $a_M^*$ and then $f_{u_1\cdots u_\ell}([w])=X$. This violates our assumption that $|a_i|=|w|$.

			So the only other possibility is that $|a_i|=|w|$ and
			for any $M\in\N$, there is an $i>M$ where $0\leq \beta_{i}<\beta_{i+1}\leq |a_i|$ which is clearly not possible.	
		\end{proof}
	\end{enumerate}
\end{example}

Next we give an example whose any map in the IFS is  exact as a conventional dynamical system, though the IFS itself is not exact; somehow presenting opposite properties comparing to the previous example.
\begin{example}\label{ex maps exact} 
	Let $\mathfrak{I}=(X=[0,\,1],\,\{f_0,\,f_1\},\,\Sigma=\{(01)^\infty,\,(10)^\infty\})$. To define $f_i$,  choose two different points $x_0,\,x_1\in (0,\,1)$ and small open interval $I_i$ around  $x_i$ such that $x_j\not\in I_i$ if $j\neq i$.  We aim to have
	\[
	I_0\xrightarrow{f_0}I_1\xrightarrow{f_1} I_0 
	\]
	and $f_0$ (resp. $f_1$) being contracting on  points of $I_0$ (resp. $I_1$) with an infimum rate  $c_0\in (\frac{1}{2},\,1)$ and elsewhere expansive with infimum rate $e_0>2$.
	An example of $ f_0 $ and $ f_1 $ can be those presented in Figure \ref{figure 2exacts}.
	
	\tikzset{every picture/.style={line width=0.75pt}} 
	\begin{figure}[ht]
		\begin{tikzpicture}[x=0.87pt,y=0.87pt,yscale=-1,xscale=1]
			
			\draw  (30.73,320) -- (210,320)(50,176.98) -- (50,339.11) (203.19,314.89) -- (210.19,320) -- (203.19,324.89) (44.73,183.98) -- (50,176.98) -- (54.73,183.98);
			\draw  [dash pattern={on 0.84pt off 2.51pt}]  (50,210) -- (162.57,210) ;
			\draw  [dash pattern={on 0.84pt off 2.51pt}]  (162,210) -- (162,320) ;
			\draw  [dash pattern={on 0.84pt off 2.51pt}]  (105,210) -- (105,320) ;
			\draw  [dash pattern={on 0.84pt off 2.51pt}]  (134,210) -- (134,320) ;
			\draw  [dash pattern={on 0.84pt off 2.51pt}]  (78,210) -- (78,320) ;
			\draw    (50,210) -- (60,266) ;
			\draw    (72,273) -- (60,266) ;
			\draw    (72,273) -- (78,320) ;
			\draw   (78,320) -- (105,210) ;
			\draw [line width=0.75]   (105,210) -- (134,320) ;
			\draw   (134,320) -- (162,210) ;
			\draw [line width=0.75]    (95.5,316.5) -- (95.5,322.5) ;
			\draw [line width=0.75]    (70.5,316.5) -- (70.5,322.5) ;
			\draw [line width=0.75]    (52.5,273) -- (47.5,273) ;
			\draw [line width=0.75]    (52.5,300) -- (47.5,300) ;

			\draw  (221.41,320) -- (400.87,320)(240,176.98) -- (240,339.11) (393.87,314.89) -- (400.87,320) -- (393.87,324.89) (235.41,183.98) -- (240,176.98) -- (245.41,183.98)  ;
			\draw  [dash pattern={on 0.84pt off 2.51pt}]  (240,210) -- (353,210) ;
			\draw  [dash pattern={on 0.84pt off 2.51pt}]  (353,210) -- (353,320) ;
			\draw  [dash pattern={on 0.84pt off 2.51pt}]  (296,210) -- (296,320) ;
			\draw  [dash pattern={on 0.84pt off 2.51pt}]  (324,210) -- (324,320) ;
			\draw  [dash pattern={on 0.84pt off 2.51pt}]  (268,210) -- (268,320) ;
			\draw [line width=0.75]  (240,320) -- (268,210) ;
			\draw    (268,210) -- (277,297) ;
			\draw    (296,320) -- (289,302) ;
			\draw   (277,297) -- (289,302) ;
			\draw   (324,210) -- (296,320) ;
			\draw   (353,320) -- (324,210) ;
			\draw [line width=0.75]    (284.5,317.5) -- (284.5,323.5) ;
			\draw [line width=0.75]    (260,317.5) -- (260,322.5) ;
			\draw [line width=0.75]    (242.5,273) -- (237.5,273) ;
			\draw [line width=0.75]    (242.5,300) -- (236.5,300) ;
			
			\draw (30,205) node [anchor=north west][inner sep=0.75pt]    {$1$};
			\draw (74,321) node [anchor=north west][inner sep=0.75pt]    {\tiny $\dfrac{1}{4}$};
			\draw (129,321) node [anchor=north west][inner sep=0.75pt]    {\tiny $\dfrac{3}{4}$};
			\draw (158,325) node [anchor=north west][inner sep=0.75pt]    {$1$};
			\draw (28,297) node [anchor=north west][inner sep=0.75pt]    {$x_{0}$};
			\draw (28,268) node [anchor=north west][inner sep=0.75pt]    {$x_{1}$};
			\draw (60,325) node [anchor=north west][inner sep=0.75pt]    {$x_{0}$};
			\draw (90,325) node [anchor=north west][inner sep=0.75pt]    {$x_{1}$};
			\draw (220,205) node [anchor=north west][inner sep=0.75pt]    {$1$};
			\draw (264,321) node [anchor=north west][inner sep=0.75pt]    {\tiny $\dfrac{1}{4}$};
			\draw (319,321) node [anchor=north west][inner sep=0.75pt]    {\tiny $\dfrac{3}{4}$};
			\draw (350,325) node [anchor=north west][inner sep=0.75pt]    {$1$};
			\draw (217,297) node [anchor=north west][inner sep=0.75pt]    {$x_{0}$};
			\draw (217,268) node [anchor=north west][inner sep=0.75pt]    {$x_{1}$};
			\draw (251,325) node [anchor=north west][inner sep=0.75pt]    {$x_{0}$};
			\draw (281,325) node [anchor=north west][inner sep=0.75pt]    {$x_{1}$};
			\draw (92,345) node [anchor=north west][inner sep=0.75pt]    {$f_{0}$};
			\draw (283,345) node [anchor=north west][inner sep=0.75pt]    {$f_{1}$};
			\draw (90,195) node [anchor=north west][inner sep=0.75pt]    {\tiny $( \frac{1}{2},1)$};
			\draw (280,195) node [anchor=north west][inner sep=0.75pt]    {\tiny $( \frac{1}{2},1)$};
			
		\end{tikzpicture}
		\caption{$([0,\,1],\,f_i)$ is exact, but
			$\mathfrak{I}=([0,\,1],\,\{f_0,\,f_1\},\,\{(01)^\infty,\,(10)^\infty\})$ is not even {point transitive.}}\label{figure 2exacts}
	\end{figure}
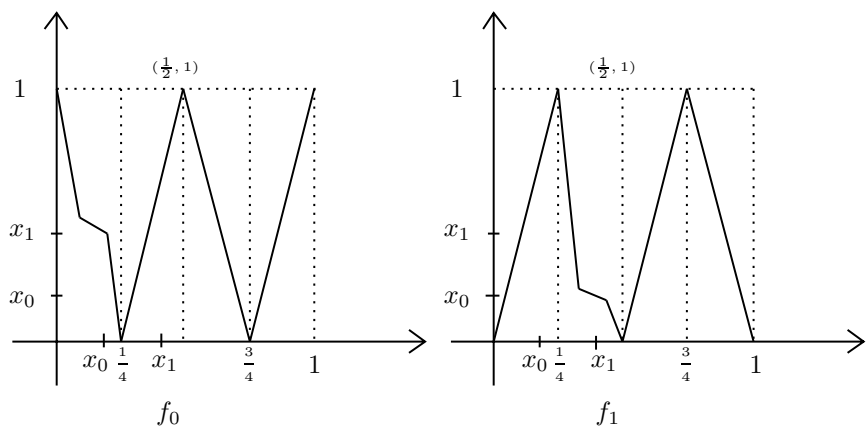
	
	This construction guarantees that $f_i$ being exact; however, a sufficiently small neighborhood around $x_0$ shrinks to a point along  $\sigma=(01)^\infty$. Thus $\mathfrak{I}$ cannot be  exact.
\end{example}



\end{document}